\documentclass[11pt]{amsart}
\usepackage{epsfig}
\usepackage{graphicx}
\usepackage{amscd}
\usepackage{amsmath}
\usepackage{amsxtra}
\usepackage{amsfonts}
\usepackage{amssymb}

\newtheorem{theorem}{Theorem}[section]

\newtheorem{lemma}[theorem]{Lemma}

\theoremstyle{definition}
\newtheorem{definition}[theorem]{Definition}
\newtheorem{remark}[theorem]{Remark}

\theoremstyle{remark}

\renewcommand{\theclaim}{\textup{\theclaim}}

\numberwithin{equation}{section}

\def\openone

{\mathchoice

{\hbox{\upshape \small1\kern-3.3pt\normalsize1}}

{\hbox{\upshape \small1\kern-3.3pt\normalsize1}}

{\hbox{\upshape \tiny1\kern-2.3pt\SMALL1}}

{\hbox{\upshape \Tiny1\kern-2pt\tiny1}}}

\makeatletter

\newbox\ipbox

\newcommand{\diracb}[1]{\left\langle #1\mathrel{\mathchoice

{\setbox\ipbox=\hbox{$\displaystyle \left\langle\mathstrut
#1\right.$}

\vrule height\ht\ipbox width0.25pt depth\dp\ipbox}

{\setbox\ipbox=\hbox{$\textstyle \left\langle\mathstrut
#1\right.$}

\vrule height\ht\ipbox width0.25pt depth\dp\ipbox}

{\setbox\ipbox=\hbox{$\scriptstyle \left\langle\mathstrut
#1\right.$}

\vrule height\ht\ipbox width0.25pt depth\dp\ipbox}

{\setbox\ipbox=\hbox{$\scriptscriptstyle \left\langle\mathstrut
#1\right.$}

\vrule height\ht\ipbox width0.25pt depth\dp\ipbox}

}\right. }

\newcommand{\dirack}[1]{\left. \mathrel{\mathchoice

{\setbox\ipbox=\hbox{$\displaystyle \left.\mathstrut
#1\right\rangle$}

\vrule height\ht\ipbox width0.25pt depth\dp\ipbox}

{\setbox\ipbox=\hbox{$\textstyle \left.\mathstrut
#1\right\rangle$}

\vrule height\ht\ipbox width0.25pt depth\dp\ipbox}

{\setbox\ipbox=\hbox{$\scriptstyle \left.\mathstrut
#1\right\rangle$}

\vrule height\ht\ipbox width0.25pt depth\dp\ipbox}

{\setbox\ipbox=\hbox{$\scriptscriptstyle \left.\mathstrut
#1\right\rangle$}

\vrule height\ht\ipbox width0.25pt depth\dp\ipbox}

} #1\right\rangle}

\newcommand{\Ru}{\mathbb{R}}

\newcommand{\bz}{\mathbb{Z}}

\newcommand{\br}{\mathbb{R}}

\def\blfootnote{\xdef\@thefnmark{}\@footnotetext}

\renewcommand{\mod}{\operatorname{mod}}

\input xy
\xyoption{all}
\usepackage{amssymb}

\def\T{\mathbb{T}}

\def\-{^{-1}}

\def\D{\mathcal{D}}

\def\T{\mathcal{T}}

\def\S{\mathcal{S}}

\begin{document}
\title[Bessel sequences of exponentials on fractal measures]{Bessel sequences of exponentials on fractal measures}
\author{Dorin Ervin Dutkay}
\blfootnote{}
\address{[Dorin Ervin Dutkay] University of Central Florida\\
    Department of Mathematics\\
    4000 Central Florida Blvd.\\
    P.O. Box 161364\\
    Orlando, FL 32816-1364\\
U.S.A.\\} \email{ddutkay@mail.ucf.edu}

\author{Deguang Han}
\address{[Deguang Han]University of Central Florida\\
    Department of Mathematics\\
    4000 Central Florida Blvd.\\
    P.O. Box 161364\\
    Orlando, FL 32816-1364\\
U.S.A.\\} \email{dhan@mail.ucf.edu}

\author{Eric Weber}
\address{[Eric Weber]Department of Mathematics\\
396 Carver Hall\\
Iowa State University\\
Ames, IA 50011\\
U.S.A.\\} \email{esweber@iastate.edu}

\thanks{}

\begin{abstract}
Jorgensen and Pedersen have proven that a certain fractal measure
$\nu$ has no infinite set of complex exponentials which form an orthonormal
set in $L^2(\nu)$.  We prove that any fractal measure $\mu$ obtained from an
affine iterated function system possesses a sequence of complex exponentials which
forms a Riesz basic sequence, or more generally a Bessel sequence, in $L^2(\mu)$ such
that the frequencies have positive Beurling dimension.
\end{abstract}
\maketitle \tableofcontents

\section{Introduction}
In \cite{JP98}, Jorgensen and Pedersen prove two surprising results:
\begin{enumerate}
\item there exists a singular Borel probability measure $\sigma$ such that there exists a sequence $\{\lambda_{n}\}_{n=0}^{\infty} \subset \Ru$
such that the functions $e_{\lambda_{n}}(x) := e^{2 \pi i \lambda_{n} x}$ is an orthonormal basis for $L^2(\sigma)$; \item The Hausdorff measure
$\nu$ on the middle third Cantor set has the following property: for any three $\{ \lambda_{1}, \lambda_{2}, \lambda_{3} \} \subset \Ru$, the
set $\{e_{\lambda_{1}}, e_{\lambda_{2}}, e_{\lambda_{3}} \} \subset L^2(\nu)$ is not orthogonal.
\end{enumerate}
In both cases, the measure
arises as the (unique) invariant measure under an iterated function system \cite{Hut81}.
 We prove, in contradistinction to item (ii) above, that every measure $\mu$ arising from
 a suitable iterated function system possesses an infinite sequence $\{\lambda_{n} \}_{n=0}^{\infty}$ such
 that the sequence $\{ e_{\lambda_{n}}\}_{n=0}^{\infty}$ is a Riesz basic sequence in $L^2(\mu)$.  Moreover,
 this sequence has positive Beurling dimension.

Frames were introduced by Duffin and Schaeffer \cite{DuSc52} in the context of nonharmonic Fourier series, and today they have applications in a
wide range of areas. Frames provide robust, basis-like representations of vectors in a Hilbert space. The potential redundancy of frames often
allows them to be more easily constructible than bases, and to possess better properties than those that are achievable using bases. For
example, redundant frames offer more resilience to the effects of noise or to erasures of frame elements than bases. Following Duffin and
Schaeffer a Fourier frame or frame of exponentials is a frame of the form $\{e^{2\pi i\lambda·x}\}_{\lambda\in\Lambda}$ for the Hilbert space
$L^2[0, 1]$. Fourier frames are also closely connected with sampling sequences or complete interpolating sequences \cite{OSANN}.

\begin{definition}
A sequence $\{x_n\}_{n=1}^{\infty}$ in a Hilbert space (with inner product $\langle \cdot , \cdot \rangle $) is \emph{Bessel} if there exists a positive constant $B$ such that
\[ \sum_{n=1}^{\infty} | \langle v , x_n \rangle |^2 \leq B \|v\|^2. \]
This is equivalent to the existence of a positive constant $D$ such that for every finite sequence $\{c_{1}, \dots , c_{K} \}$ of complex numbers
\[ \| \sum_{n=1}^{K} c_{n} x_n \| \leq D \sqrt{\sum_{n=1}^{K} |c_{n}|^2}. \]
Here $D^2 = B$ is called the Bessel bound.

The sequence is a frame if in addition to being a Bessel sequence there exists a positive constant $A$ such that
\[ A \| v\|^2 \leq \sum_{n=1}^{\infty} | \langle v , x_n \rangle |^2 \leq B \|v\|^2. \]
In this case, $A$ and $B$ are called the lower and upper frame bounds, respectively.

The sequence is a Riesz basic sequence if in addition to being a Bessel sequence there exists a positive constant $C$ such that for every finite sequence $\{c_{1}, \dots , c_{K} \}$ of complex numbers
\[ C \sqrt{\sum_{n=1}^{K} |c_{n}|^2} \leq \| \sum_{n=1}^{K} c_{n} x_n \| \leq D \sqrt{\sum_{n=1}^{K} |c_{n}|^2}. \]
Here $C$ and $D$ are called the lower and upper basis bounds, respectively.

\end{definition}
The main result of Duffin and Schaeffer is a sufficient density condition for $\{e^{2\pi i\lambda\cdot x}\}_{\lambda\in \Lambda}$ to be a frame for $L^2[0,1]$. Landau \cite{MR0222554}, Jaffard
\cite{Jaffard} and Seip \cite{Seip2} ``almost"  characterize the frame properties of $\{e^{2\pi i\lambda\cdot x}\}_{\Lambda\in \Lambda}$ in
terms of lower Beurling density:
$$
\mathcal D^-(\Lambda):= \liminf_{h\rightarrow\infty}\inf_{x\in \br}\frac{\#(\Lambda\cap[x-h,x+h])}{2h}.
$$

\begin{theorem}\label{th1.1} For $\{e^{2\pi i\lambda\cdot x}\}_{\Lambda\in \Lambda}$ to be a frame for $L^2[0,1]$, it is necessary that $\Lambda$ is relatively separated
and $\mathcal D^-(\Lambda)\geq 1$,  and it is sufficient that $\Lambda$ is relatively separated and $\mathcal D^-(\Lambda)> 1$.
\end{theorem}

The property of relative separation is equivalent to the condition that the upper Beurling density $$\mathcal D^{+}(\Lambda):=
\limsup_{h\rightarrow\infty}\sup_{x\in R}\frac{\#(\Lambda\cap[x-h,x+h])}{2h}$$ is finite.

For the critical case when $\mathcal D^-(\Lambda)= 1$, the complete characterization  was beautifully formulated by Joaquim Ortega-Cerd\`{a} and Kristian
Seip in \cite{OSANN} where the key step was to connect the problem to de Branges' theory of Hilbert spaces of entire functions, and this
new characterization lead to applications in a classical inequality of H. Landau and an approximation problem for subharmonic functions.

In recent years there has been a wide range of interests in expanding the classical Fourier analysis to fractal or more general probability measures \cite{MR2509326,MR2435649,MR1744572,MR1655831,MR2338387,MR2200934,MR2297038,MR1785282,MR2279556,MR2443273}. One of the central themes of this area of research involves constructive and computational bases in $L^2(\mu)$, where $\mu$ is a measure which is determined by some self-similarity property. These include classical Fourier bases, as well as wavelet and frame constructions.

For $L^2[0,1]$, a sequence of exponentials is Bessel if the frequency set $\Lambda$ has finite upper Beurling density.  For a singular measure $\nu$, a necessary (but not sufficient) condition for such a sequence to be Bessel in $L^2(\nu)$ is that the upper Beurling density of $\Lambda$ is 0 (see \cite{DHSW10}). Since the measures we consider here are singular, we shall use \emph{Beurling dimension} as a replacement for Beurling density.

\begin{definition}\label{deff3}
\cite{CKS08} Let $\Lambda$ be a discrete subset of $\br^d$. For $r>0$, the {\it upper Beurling density corresponding to $r$} (or {\it $r$-Beurling density}) is defined by
$$\mathcal D_r^+(\Lambda):=\limsup_{h\rightarrow\infty}\sup_{x\in\br^d}\frac{\#(\Lambda\cap(x+h[-1,1]^d))}{h^r}.$$
The {\it upper Beurling dimension} (or simply the {\it Beurling dimension}) is defined by
$$\dim^+(\Lambda):=\sup\{r>0 : \D^+_r(\Lambda)>0\} = \inf\{r>0 : \D_r^+(\Lambda)<\infty\}.$$

Given a set of exponential functions $E(\Lambda):=\{e_\lambda : \lambda\in\Lambda\}$ we also say that $\mathcal D_r^+(\Lambda)$ is the $r$-Beurling density of $E(\Lambda)$.
\end{definition}

\begin{definition}\label{defaifs}
Let $R$ be a $d\times d$ expansive integer matrix, $B\subset\bz^d$, with $\#B=N\geq2$. Define the iterated function system
$$\tau_b(x)=R^{-1}(x+b),\quad(x\in\br^d).$$

For convenience, we let $S := R^{T}$.

Let $(p_b)_{b\in B}$ be a finite set of probabilities, i.e., $0<p_b<1$, $\sum_{b\in B}p_b=1$.
Define the following operator $\mathcal T$ on Borel probability measures on $\br^d$
\begin{equation}\label{eqtmu1}
(\T\gamma)(E)=\sum_{b\in B}p_b\gamma(\tau_b^{-1}(E)),
\end{equation}
for all Borel sets $E$.
Equivalently the measure $\T\gamma$ is defined by
\begin{equation}
\int f\,d\T\gamma=\sum_{b\in B}p_b\int f\circ\tau_b\,d\gamma,
\label{eqtmu2}
\end{equation}
for all continuous functions $f$ on $\br^d$.

We denote by $\mu := \mu_{B,p}$ the unique invariant measure for the operator $\T$, i.e. $\T\mu_{B,p}=\mu_{B,p}$, whose existence is guaranteed by \cite{Hut81}.
\end{definition}

\begin{definition}
For a Borel probability measure $\gamma$, if $\Lambda = \{\lambda_{n}\}_{n=0}^{\infty} \subset \mathbb{R}$ is such that $\{ e_{\lambda_{n}} \} \subset L^2(\gamma)$ is a Bessel sequence, we say $\Lambda$ is a Bessel spectrum for $\gamma$.  Likewise, $\Lambda$ is a Riesz basic spectrum if $\{ e_{\lambda_{n}} \}$ is a Riesz basic sequence in $L^2(\gamma)$.
\end{definition}

In the classical Lebesgue measure case, it is relatively easy (with the help of Theorem \ref{th1.1}) to construct frames/Riesz bases or more
generally Bessel sequences/Riesz sequences $\{e^{2\pi i\lambda\cdot x}\}_{\Lambda\in \Lambda}$ with $\Lambda$ having positive Burling density.
However, this is not the case anymore for fractal measures.  Indeed, for the fractal measure $\mu_{B,p}$ in the case that $R = 3$, $B =
\{0,2\}$, and $p_{0} = p_{2} = 1/2$, the corresponding measure $\mu_{3}$ (which is the Hausdorff measure on the middle third Cantor) has the property that $\{
3^{n} : n=0,1,\dots \}$ is NOT a Bessel spectrum (and hence can not be a Riesz basic spectrum) \cite[Proposition 3.10]{DHSW10}. Note that this
set $\{ 3^{n} : n=0,1,\dots \}$  is very ``sparse" and in fact it has the Beurling dimension equal to $0$. One of the open problems for the
fractal measure $\mu_{3}$ is that whether frames or Riesz bases spectrum exist. In \cite[Theorem 3.5]{DHSW10} it was proved that for a fractal
measure $\mu_{B,p}$, a necessary condition for $\Lambda$ to be a Bessel spectrum is that the Beurling dimension of $\Lambda$ is at most
$\log_{R} B$, and that the Beurling dimension of $\Lambda$ is equal to $\log_{R} B$ (under a mild technical condition) in order for $\Lambda$ to
be a frame spectrum. The above example $(\Lambda = \{ 3^{n} : n=0,1,\dots \}$ ) shows that this finite Beurling dimension condition is not
sufficient for $\Lambda$ to be even a Bessel spectrum. This naturally leads to the existence problem for Bessel spectrum  and Riesz basic
spectrum with positive Beurling dimensions. The main purpose of this paper is to prove that Bessel spectrum  and Riesz basic spectrum with
positive Beurling dimension exists for all the fractal measures $\mu_{B,p}$. We believe that this is an important positive step toward answering
the question of whether frames or Riesz bases spectra exist for the fractal measure $\mu_{3}$.

\section{Spectra of Positive Beurling Dimension}\label{bess}

We start with our main theorem:

\begin{theorem}\label{thb1}
Let $R$ be a $d\times d$ expansive integer matrix, $0\in B\subset \bz^d$ , $(p_b)_{b\in B}$ a list of probabilities and let $\mu=\mu_{B,p}$ be the invariant measure associated to the iterated function system
$$\tau_b(x)=R^{-1}(x+b),\quad(x\in\br^d, b\in B)$$
and the probabilities $(p_b)_{b\in B}$. Then $\mu$ has an infinite Riesz basic spectrum of positive Beurling dimension.
\end{theorem}

The proof proceeds via a series of lemmas.  Throughout the remainder of the paper, $R$, $B$, $p$ are fixed, and $\mu := \mu_{B,p}$.

\begin{lemma}\label{lemip}
The Fourier transform of the invariant measure $\mu$ satisfies the scaling equation
\begin{equation}
\widehat\mu(x)= m(S^{-1} x)\widehat\mu(S^{-1} x),\quad(x\in \br^d)
\label{eqsc}
\end{equation}
where
\begin{equation}
m(x):= \sum_{b\in b}p_be^{2\pi i b\cdot x},\quad(x\in\br^d)
\label{eqmb}
\end{equation}

The function $\widehat\mu$ is given by the infinite product formula
\begin{equation}
\widehat\mu(x)=\prod_{k=1}^\infty m\left(S^{-k}x\right),\quad(x\in\br^d)
\label{eqip}
\end{equation}

The infinite product converges uniformly on compact subsets.
\end{lemma}

\begin{proof}
Apply the Fourier transform to the invariance equation \eqref{eqtmu2}. See e.g. \cite{JP98a,DJ06b} for details.
\end{proof}

\begin{lemma}\label{lemb1}
 There exists $p \in \mathbb{N}$, $0<\rho<1$, and a finite set $A\subset\bz^d\setminus\{0\}$ with $\#A \geq 2$ such that if
$M:=\max\{\|S^{-p}a-S^{-p}a'\| : a,a'\in A\cup\{0\}, a\neq a'\}$, then
\begin{equation} \label{eqmrho}
\left|m\left(S^{-p} (a-a')+x\right)\right|\leq \rho
\end{equation}
for all $x$ with $\|x\|\leq \frac{M\|S^{-1}\|^{p}}{1-\|S^{-1}\|^p}$,  and for all $a,a'\in A\cup \{0\}$ with $a\neq a'$.

In addition, the elements of $A$ are incongruent $\mod S^p\bz^d$.
\end{lemma}

\begin{proof}
We have to stay away from the points where $|m|$ is 1. Note that $|m(x)|=1$ implies
$$\left|\sum_{b\in B}p_be^{2\pi ib\cdot x}\right|=1.$$
Since $0\in B$, the term $p_0\cdot 1$ appears in the sum. Using the triangle inequality we must have $e^{2\pi ib\cdot x}=1$ for all $b\in B$ and therefore $b\cdot x\in\bz$. It follows that $$\S_1:=\{x : |m(x)|=1\}=\{x : b\cdot x\in\bz\}.$$

Since $\S_1$ has Lebesgue measure zero in $\br^d$, we can find two distinct points $x_0,x_1$ such that $\|x_0\|,\|x_1\|\leq 1$ with $\pm x_0,\pm x_1,\pm (x_0-x_1)\not\in\S_1$.

Choose $\delta'<1$ and let $c:=\|S^{-1}\|<1$ (since the matrix $S$ is expansive). We can pick $p \in \mathbb{N}$ large enough such that
\begin{equation} \label{eqpf1}
\frac{4c^p}{1-c^p}<\delta'.
\end{equation}
Moreover, since the volume of the lattice $S^{-p} \bz^{d}$ goes to $0$ as $p$ gets large, we may choose $p$ so that additionally there exist integers $a_0\neq a_1\in\bz^d$ with $\|x_0-S^{-p} a_0\|, \|x_1-S^{-p}a_1\|\leq \delta'$.

Let $A:=\{a_0,a_1\}$. Then $\|S^{-p}a_i\|\leq 1+\delta'<2$ so $M$ as defined in the hypothesis will be  less than $4$.  If $\|y\|<M c^p/(1-c^p)$ then for $a,a'\in A\cup\{0\}$, $a\neq a'$, there exists $x,x'\in \{x_0,x_1,0\}$ such that
$$\|(S^{-p}(a-a')+y)-(x-x')\|\leq \|S^{-p}a-x\|+\|S^{-p}a'-x'\|+\|y\|\leq 3\delta'.$$
Thus, if $\delta'$ is small enough, $S^{-p}(a-a')+y$, being close to $x-x'$, stays away from the set $\S_1$ so by uniform continuity of $m$, there is a $\rho<1$ such that
$$|m(S^{-p}(a-a')+y)|\leq \rho,$$
for all $y$ with $\|y\|\leq Mc^p/(1-c^p)$. This proves the existence of $p,\rho$ and $A$.

The elements in $A\cup\{0\}$ cannot be congruent $\mod S^p\bz^d$ because $|m(S^{-p}(a-a'))|<1$, while $m(k)=1$ for $k\in\bz^d$ and $m$ is $\bz^d$ periodic.

\end{proof}

\begin{definition}\label{defb1}

Let $p,\rho$ and $A$ be as in Lemma \ref{lemb1}.  Let
$$\Lambda(A,p) := \{ a_0+S^pa_1+\dots+S^{pr}a_r : a_{i} \in A \cup \{0\} \}.$$
We identify the integer $a_0+S^pa_1+\dots+S^{pr}a_r$ with the word $a_0a_1\dots a_r$ and with the infinite word $a_0a_1\dots a_r a_{r+1}\dots$, where $a_i=0$ for $i\geq  r+1$.
Since the elements in $A\cup\{0\}$ are incongruent $\mod S^p\bz^d$, different digits means different integers.

For two such $\lambda=a_0a_1\dots$, $\lambda'=a_0'a_1'\dots$, we define the Hamming distance between them by
$$d_p(\lambda,\lambda')=\#\{ i : a_i\neq a_i'\}.$$

\end{definition}

\begin{lemma}\label{lemb2}
Let $p,\rho,A,M$ be as in Lemma \ref{lemb1}. Let $\lambda=a_0a_1\dots$, $\lambda'=a_0'a_1'\dots$ be distinct words with digits in $A$.
Then
$$|\widehat\mu(\lambda-\lambda')|\leq \rho^{d_p(\lambda,\lambda')}.$$
\end{lemma}

\begin{proof}
We use the infinite product formula for $\widehat\mu$, and we group every $p$ terms. We have
$$\widehat\mu(x)=\prod_{n=1}^\infty m^{(p)}(S^{-np}x),\quad(x\in\br^d)$$
where
$$m^{(p)}(x)=m(x)m(Sx)\dots m(S^{p-1}x).$$

Since $|m|\leq 1$, we get for any $I \subset \mathbb{N}$,
\begin{equation}
|\widehat\mu(x)| \leq \prod_{n \in I} |m^{(p)}(S^{-np}x)|.
\label{eqb3}
\end{equation}

Suppose $a_{n} \neq a_{n}'$.  For $k \geq n$, $S^{-np}(S^{kp}(a_{k} - a_{k}')) \in \mathbb{Z}^d$.  Therefore, we have
$$S^{-np}(\lambda-\lambda')\equiv S^{-p}(a_{n-1}-a_{n-1}')+S^{-2p}(a_{n-2}-a_{n-2}')+\dots+S^{-np}(a_0-a_0')\mod\bz^d,$$
and since
$$\|S^{-p}(a_{n-1}-a_{n-1}')+\dots+S^{-np}(a_0-a_0')\|\leq \|S^{-p}\|(M+\|S^{-p}\|\cdot M+\cdots + \|S^{(-n-1)p}\| \cdot M)\leq \frac{M\|S^{-1}\|^p}{1-\|S^{-1}\|^p}$$
with Lemma \ref{lemb1} we obtain
$$|m(S^{-np}(\lambda-\lambda'))|\leq \rho.$$

Thus, using \eqref{eqb3}, we obtain
$$|\widehat\mu(\lambda-\lambda')|\leq \prod_{n \in I} |m^{(p)}(S^{-np}(\lambda - \lambda')| \leq \rho^{d_p(\lambda,\lambda')},$$
where $I := \{ n : a_{n} \neq a_{n}' \}$ with $\# I = d_p(\lambda,\lambda')$.
\end{proof}

\begin{remark}\label{rem1}
By changing the set of digits $A$ and $p$ we can assume that $\rho$ is as small as we want. This is because we can replace each digit in $A$ by a repetition of it, say $l$ times. So, for example $12$ is replaced by $111222$, where $l=3$. By doing this the distance between any two words is multiplied by $l$. So if we replace $p$ by $p\cdot l$ and $A$ by $A^{(l)}:=\{a^{(l)}:=\underbrace{aa\dots a}_{l\mbox{ times}} : a\in A\}$, the number $\rho$ in Lemma \ref{lemb2} is replaced by $\rho^l$, which can be made as small as needed. More precisely, we have
$$d_{p,A}(\lambda^{(l)},\gamma^{(l)})=l\cdot d_{pl,A^{(l)}}(\lambda,\gamma),\mbox{ for }\lambda,\gamma\in A^{(l)}.$$
The distance $d_{pl,A^{(l)}}$ counts each digit $a^{(l)}$ as just one.

\end{remark}

\begin{lemma}\label{lemb3}
Let $p,\rho,A$ be as in Lemma \ref{lemb1}, and let $\Lambda \subset \Lambda(A,p)$.  Suppose
$$C:=\sup_{\lambda\in\Lambda}\sum_{\lambda'\in\Lambda\setminus\{\lambda\}}\rho^{d_p(\lambda,\lambda')}<\infty.$$
Then $\Lambda$ is a Bessel spectrum with Bessel bound $1+C$.

Moreover, if $C<1$, then $\Lambda$ is a Riesz basic spectrum.
\end{lemma}

\begin{proof}
Using Lemma \ref{lemb2}, we have for all $\lambda\in\Lambda$:
$$\sum_{\lambda'\in\Lambda}|\widehat\mu(\lambda-\lambda')|\leq 1+\sum_{\lambda\neq\lambda'}\rho^{d_p(\lambda,\lambda')}\leq 1 + C.$$
An application of Schur's lemma shows that the Grammian of the set $\{e_\lambda : \lambda\in\Lambda\}$ is bounded, and, if $C<1$, it is also diagonally dominant, hence invertible (see \cite[Proposition 3.5.4]{Chr03}). This implies that $\Lambda$ is a Bessel spectrum with bound $1+C$ and is a Riesz basic spectrum when $C<1$.
\end{proof}

\begin{lemma}\label{lemb4}
Let $A$ be an alphabet with $2$ letters. Then there exists $k_{0}\geq 1$ such that for every  $k\geq k_{0}$ and every $n$  there is
 set $\Lambda_n$ containing $2^{n}$ words of length $kn$, i.e., $\Lambda_n\subset A^{kn}$,
 such that the Hamming distance between any two distinct words in $\Lambda_n$ is at least $n$.

\end{lemma}
\begin{proof} This is a consequence of the Gilbert-Varshamov  bound \cite{Lin99} which states that if $A_{q}(m, d)$ is the maximum possible size of a
$q$-ary code $C$ with length $m$ and minimum Hamming distance $d$ (a $q$-ary code is a code over the field ${\mathbb F}_{q}$ with $q$-elements), then
$$
A_{q}(m, d) \geq \frac{q^{m}}{\sum_{j=0}^{d-1}C_{m}^{j}(q-1)^{j}},
$$
where $C_{m}^{j}$ are the binomial coefficients.

In our case,  $q=2$, $d = n$ and $m = kn$. Let $H(x) = -x\log_{2}x - (1-x)\log_{2}(1-x)$. Using the argument in \cite{BIPW10} we can show that for $k > 2$
$$
 \sum_{j=0}^{n}C_{m}^{j} < 2^{H(\frac{1}{k})kn}.
$$
In fact, let $u = \frac{1}{k}$. Then
$$
2^{-H(u)} = \left(\frac{u}{1-u}\right)^{u}(1 - u)
$$
and so
$$
1 = (u + (1-u))^{kn} > \sum_{j=0}^{n}C_{kn}^{j}u^{j}(1-u)^{kn-j}= \sum_{j=0}^{n}C_{kn}^{j}\left(\frac{u}{1-u}\right)^{j}(1-u)^{kn}$$
$$> \sum_{j=0}^{n}C_{kn}^{j}\left(\frac{u}{1-u}\right)^{n}(1-u)^{kn}
  = \sum_{j=0}^{n}C_{kn}^{j}\left[\left(\frac{u}{1-u}\right)^{u}(1-u)\right]^{kn}= 2^{-H(u)kn}\sum_{j=0}^{n}C_{kn}^{j}$$

Thus we get $$
 \sum_{j=0}^{n}C_{kn}^{j} < 2^{H(\frac{1}{k})kn},
$$
hence
$$
A_{2}(kn, n) \geq \frac{2^{kn}}{\sum_{j=0}^{n-1}C_{kn}^{j}} \geq  \frac{2^{kn}}{\sum_{j=0}^{n}C_{kn}^{j}}\geq 2^{(1-H(\frac{1}{k}))kn}.
$$
Since $1-H(\frac{1}{k})\rightarrow1$ as $k\rightarrow\infty$, there exists a $k_0$ such that for $k\geq k_0$, $(1-H(\frac{1}{k}))k\geq 1$ and
$$
A_{2}(kn, n) \geq 2^n.
$$
\end{proof}

\begin{proof}[Proof of Theorem \ref{thb1}]
To complete the proof of Theorem \ref{thb1}, we construct a set $\Lambda$ that satisfies the hypothesis of Lemma \ref{lemb3}.

Let $A,p,\rho$ as in Lemma \ref{lemb1} where $A$ has two non-zero elements.  Without loss of generality, by Remark \ref{rem1} we may assume that $\rho < 1/4$.

We choose a sequence $q_1,q_2,\dots $ of natural numbers such that $q_1+\dots+q_{n-1}+1\leq q_n$, for all $n$; for example $q_n=2^n$.

Let $k\geq k_{0}$ where $k_{0}$ is as in Lemma \ref{lemb4}, let $\Lambda_{n}$ be the set of words of $A^{kq_n}$ guaranteed by Lemma \ref{lemb4} with at least $2^{q_n}$ elements and the Hamming distance between any two words is at least $q_n$.

We define $\Lambda$ by concatenating words in $\Lambda_i$ as follows
$$\Lambda:=\{\lambda_1\dots\lambda_n : \lambda_i\in \Lambda_i, n\geq 1\}.$$

Fix $\lambda=\lambda_1\dots\lambda_N$; we wish to estimate
$$ \sum_{\lambda' \in \Lambda \setminus \{\lambda\}} \rho^{d_{p}(\lambda, \lambda')}. $$

For any $\lambda'=\lambda_1'\dots\lambda_m' \in \Lambda$, there exists a natural number $r(\lambda')$ which is the largest index such that $\lambda_r\neq \lambda_r'$. The Hamming distance between $\lambda_{r}$ and $\lambda'_{r}$ is at least $q_{r(\lambda')}$, so the Hamming distance between $\lambda$ and $\lambda'$ is also at least $q_{r(\lambda')}$. Thus, for a fixed $r_{0} \in \mathbb{N}$, we count how many $\lambda' \in \Lambda$ which have $r(\lambda') = r_{0}$.  If $r_{0} \leq N$, the number of possibilities is

$$2^{q_1}+2^{q_1+q_2}+\dots+2^{q_1+\dots+q_{r_{0}}}\leq \sum_{i=0}^{q_1+\dots+q_{r_{0}}}2^{i}\leq 2^{q_1+\dots+q_{r_{0}}+1}\leq 2^{2q_{r_{0}}}=4^{q_{r_{0}}}.$$

If $r_{0}>N$, then $\lambda_{n+1} = \dots = \lambda_{r_{0}} = 0 = \lambda_{r_{0}+1}'$ with $\lambda_{r_{0}}' \neq 0$, so the number of possibilities is at most
$$2^{q_1+\dots+q_{r_{0}}}\leq 2^{2q_{r_{0}}}=4^{q_{r_{0}}}.$$

It follows that
$$\sum_{\lambda'\in\Lambda\setminus\{\lambda\}}\rho^{d_p(\lambda,\lambda')}\leq \sum_{r=1}^\infty 4^{q_r}\rho^{q_r}.$$
Since $\rho<1/4$, this sum converges.

We now associate words in $\Lambda$ with integers written in base $S^{p}$ with coefficients from the words in $\Lambda$ as in Definition \ref{defb1}.  Combining Lemmas \ref{lemb2} and \ref{lemb3} we conclude that $\Lambda$ is a Bessel spectrum. Taking $q_1$ larger if needed, we can get the sum to be less than 1, so we obtain a Riesz basic spectrum.

It remains to prove that the Beurling dimension is positive.  We will use the following lemma that can be obtained by a straightforward computation.
\begin{lemma}\label{lembe}
To compute the Beurling dimension, the unit cube $[-1,1]^d$ in Definition \ref{deff3} can be replaced by any bounded set $Q$ that contains $0$ in the interior.
\end{lemma}

Consider the elements in $\Lambda_1\dots\Lambda_n$. There are at least $2^{q_n}$ such elements, since just $\Lambda_n$ has $2^{q_n}$ elements.
The length of such a word is $kq_1+\dots+kq_n\leq 2kq_n-1$. Let $C=\max_{a\in A}\|a\|$. Then the integer represented by this word will have absolute value less than
$$C+\|S\|^pC+\dots+\|S\|^{p\cdot (2kq_n-1)}C\leq D \|S\|^{2kq_np}$$
for some constant $D$.

Therefore, in the ball of radius $D\|S\|^{2kq_np}$ there are at least $2^{q_n}$ elements in $\Lambda$. Using $Q=B(0,D)$ in Lemma \ref{lembe}, $x=0$, and $h=\|S\|^{2kq_np}$ in the Definition \ref{deff3}, this implies that the Beurling dimension of $\Lambda$ is at least
$\log_{\|S\|^{2kp}}2>0$.
\end{proof}

We conclude with a remark concerning the usual Cantor middle third set.  This set, and its invariant measure, is generated by the iterated function system with parameters $R = 3$, $B = \{0, 2\}$, and $p_{0} = p_{2} = \frac{1}{2}$.  The invariant measure $\mu_{B,p}$ for these parameters is the measure $\nu$ mentioned in item $(ii)$ in the introduction; for this measure there are no three pairwise orthogonal complex exponentials, and hence this measure possesses no orthonormal sequence of exponentials \cite{JP98}.  However, by applying Theorem \ref{thb1}, this measure on the Cantor set does possess a Riesz basic sequence of complex exponentials.

\newpage
\bibliographystyle{alpha}
\bibliography{spectral}

\end{document}